\documentclass[10pt]{article}
\usepackage{amssymb,amsmath,amsthm,mathrsfs}

\newcommand{\R}{\mathbb{R}}

\newcommand{\pd}{\partial}
\newcommand{\N}{\mathbb{N}}
\newcommand{\el}{\mathcal{L}}

\newcommand{\eps}{\varepsilon}
\newtheorem{Lemma}{Lemma}
\newtheorem{Theorem}{Theorem}
\newtheorem{Proposition}{Proposition}
\newtheorem{Corollary}{Corollary}

\theoremstyle{definition}
\newtheorem{Remark}{Remark}
\newtheorem{ass}{Assumption}

\begin{document}

\title{\bf Two dimensional Berezin-Li-Yau inequalities with a correction term}
   
\author{
Hynek Kova\v{r}{\'\i}k, Semjon Vugalter and Timo Weidl
}
\date{
\begin{center} {\small
Institute of Analysis, Dynamics
and Modeling, Universit\"at Stuttgart,
PF 80 11 40, D-70569  Stuttgart, Germany.
}
\end{center}
}

%
\maketitle

\begin{abstract}
\noindent We improve the Berezin-Li-Yau inequality in dimension two by
adding a positive correction term to its right-hand side. It is also
shown that the asymptotical behaviour of the correction term is almost
optimal. This improves a previous result by Melas, \cite{Me}. 

\end{abstract}

\section{Introduction}

Let $\Omega$ be an open bounded set in $\R^d$ and let $-\Delta$ be the
Dirichlet Laplacian on $\Omega$. We denote by $\lambda_j$ the
non-decreasing sequence of eigenvalues of $-\Delta$. 
The main object of our interest in this paper is the lower bound
\begin{equation} \label{li-yau}
\sum_{j=1}^k\, \lambda_j \geq \frac{d\, C_d}{d+2}\,
\, V^{-\frac 2d}\, k^{\frac{d+2}{d}}, \qquad  C_d= (2\pi)^2
\omega_d^{-2/d},
\end{equation}
where $V$ stands for the volume of $\Omega$ and $\omega_d$ denotes the
volume of the unit ball in $\R^d$. Inequality \eqref{li-yau} was
proved in  \cite{LY}, and is commonly known as the Li-Yau
inequality. In \cite{LW} it was pointed out that \eqref{li-yau} is in
fact the Legendre transformation of an earlier result by Berezin, see
\cite{Be}. Note also that the Li-Yau inequality  yields an individual
lower bound on $\lambda_k$ in the form
\begin{equation} \label{li-yau-indiv}
\lambda_k \geq \frac{d\, C_d}{d+2}\,
\, V^{-\frac 2d}\, k^{\frac{2}{d}} \, .
\end{equation}
For further estimates on $\lambda_k$ see 
\cite{Po,Kr,La,LW}. 
It is important to compare the lower bound \eqref{li-yau}
with the asymptotical behaviour of the sum on the left-hand side,
which reads as follows:
\begin{equation} \label{weyl-second}
\sum_{j=1}^k\, \lambda_j \ = \  \frac{d\, C_d}{d+2}\,
\, V^{-\frac 2d}\, k^{\frac{d+2}{d}} + \tilde{C}_d\,
\frac{|\partial\Omega|}{V^{1+\frac 1d}}\, \, k^{1+\frac 1d} \,
+o\left( k^{1+\frac 1d}\right)  
\quad \text{as\, } k\to\infty
\end{equation}
with
$$
\tilde{C}_d = \frac{\sqrt{\pi}\, \, \Gamma\left(2+\frac
    d2\right)^{1+\frac 1d}}{(d+1)\, \Gamma\left(\frac 32+\frac d2\right)
  \Gamma(2)^{\frac 1d}}\, .
$$
The first term in the asymptotics \eqref{weyl-second} is due to
Weyl, see \cite{We}. The second term in \eqref{weyl-second} was established,
under suitable conditions on $\Omega$, in \cite{Iv,Iv2,Mel}, see also
\cite[Chap.~1.6]{SV}. 

It follows from \eqref{weyl-second} that  
the constant in \eqref{li-yau}
cannot be improved. On the other hand, since the second asymptotical
term is positive, it is natural to ask whether one might improve
\eqref{li-yau} by adding an additional positive term of lower order in
$k$ to the right-hand
side. The first step towards this goal was done by Melas, \cite{Me},
who showed that the inequality 
\begin{equation} \label{melas}
\quad \sum_{j=1}^k\, \lambda_j \geq \frac{d\, C_d}{d+2}\,
\, V^{-\frac 2d}\, k^{\frac{d+2}{d}}+ M_d\, \,
\frac{V}{I}\ k, \qquad 
I = \min_{a\in\R^2}\, \int_\Omega\, |x-a|^2\, dx \,
\end{equation}
holds true with a factor $M_d$ which depends only on the dimension.  
Note however, that the additional term
in the Melas bound does not have the order in $k$ predicted by the
second term in \eqref{weyl-second}. Moreover,
the coefficient of the second term in \eqref{weyl-second} reflects explicitly 
the effect of the boundary of $\Omega$, whereas such a dependence is
not seen in the coefficient $V/I$ of
\eqref{melas}.

Our aim is to improve
\eqref{li-yau} and \eqref{melas} by adding a
positive contribution which reflects the nature of the second term in the
asymptotic \eqref{weyl-second}. 
Recently, one of the authors, see \cite{W},
proved an analogous improved estimate on the quantity 
$$
\sum_k\, (\Lambda-\lambda_k)_+^\sigma, \qquad \sigma\geq 3/2
$$
with a remainder term which agrees, up to a constant, with the
corresponding second term in the asymptotics of $\sum_k\,
(\Lambda-\lambda_k)_+^\sigma$ as $\Lambda\to\infty$. The proof given
in \cite{W} relies on sharp Lieb-Thirring inequalities for operator
valued potentials and works only for $\sigma\geq 3/2$. Since  the
estimates treated in present paper concern the value $\sigma=1$,  
the method of \cite{W} cannot be carried over to this case. 
We will therefore develop a different approach. 

The main idea of our strategy is explained in section
\ref{correction}. It is closely related to a modified proof 
of inequality \eqref{li-yau}, which we briefly describe in section
\ref{revisited}, see also \cite[Chap.~12]{LL}.   
The main results which represent improved Li-Yau
inequalities in case $d=2$ are formulated in section
\ref{mainresult}. Since our proof includes many technical results
concerning the geometry of the boundary of $\Omega$, we
will first give its exposition for polygons, section
\ref{polygons-proof}. Finally, in section \ref{general-proof} we
extend the proof to general domains. 

To keep the presentation as short and stringed as possible, we have
decided to restrict ourselves to the case $d=2$ throughout the paper.


\section{Preliminaries}
\label{prelim}

\noindent Following notation will be employed in the text. By
$\Theta(\cdot):\R\to\R$ we denote the Heaviside function defined by
$\Theta(x)= 0$ if $x\leq 0$ and $\Theta(x)=1$ if $x >0$. For given
$t>0$ we denote by $N_t$ the number of eigenvalues of the Dirichlet-Laplacian
in $\Omega$ less than or equal to $t$. 
Finally, we will write $[s]$ for the integer
part of a real number $s$.

\subsection{Li-Yau bound revisited}
\label{revisited}
Let $\psi_j$ be the
sequence of the normalised eigenfunctions of $-\Delta$ in $\Omega$, i.e.
\begin{equation} \label{diff-eq}
-\Delta\, \psi_j  = \lambda_j\, \psi_j \quad \text{in\, } \Omega, 
\qquad \psi_j = 0 
\quad \text{on\, } \partial\Omega, \qquad \int_\Omega|\psi_j|^2 =1\, . 
\end{equation}
In order to explain the idea which will lead to an improvement of
the results by Li-Yau and Melas, it is illustrative to see how to obtain
inequalities \eqref{li-yau} and \eqref{melas} for $d=2$ (the
same arguments apply to higher dimensions as well).
Following \cite{Be,Me} we extend the eigenfunctions $\psi_j$
continuously by zero to
the whole of $\R^2$ so that they remain in $H^1(\R^2)$. Next 
introduce the following functions:
\begin{equation}
f_j(\xi) = (2\pi)^{-1}\, \int_\Omega\, e^{-ix\cdot\xi}\, \psi_j(x)\,
dx,\qquad 
F(\xi) := \sum_{j=1}^k\, |f_j(\xi)|^2 .
\end{equation}
Since $\{\psi_j\}$ is an orthonormal basis of $L^2(\Omega)$, the
Parseval identity implies that
\begin{equation} \label{maximum1}
F(\xi) = \sum_{j=1}^k\, |f_j(\xi)|^2 \, \leq \sum_{j=1}^\infty\, |f_j(\xi)|^2
\, = \, (2\pi)^{-2}\,
\int_\Omega\, \left|e^{-ix\cdot\xi}\right|^2\, dx = (2\pi)^{-2}\,
V
\end{equation}
holds for any $\xi\in\R^2$. Next we denote by $F^*(|\xi|)$ the decreasing
radial rearrangement of $F$. Using the well-known properties of the
radial rearrangement we find  
\begin{equation} \label{cond1}
\int_{\R^2}\, F^*(|\xi|)\, d\xi = \int_{\R^2}\, F(\xi)\, d\xi=  k
\end{equation}
and
\begin{equation} \label{sum}
\sum_{j=1}^k\, \lambda_j\, = \int_{\R^2}\, |\xi|^2\, F(\xi)\, d\xi
\, \geq \,  \int_{\R^2}\, |\xi|^2\, F^*(|\xi|)\, d\xi .
\end{equation}
To find a lower bound on $\sum_{j=1}^k\lambda_j\, $ it thus suffices
to find the minimiser of the functional $\int_{\R^2}\, |\xi|^2\,
F^*(|\xi|)\, d\xi$ under the conditions \eqref{maximum1} and
\eqref{cond1}. 

The result of Li and Yau can be proved using the fact,
\cite[Chap.~12]{LL}, that 
this functional is minimised by the function
\begin{equation} \label{minimizer-li-yau}
\Phi_{LY}(|\xi|) = \left\{
\begin{array}{l@{\quad \mathrm{} \quad }l}
(2\pi)^{-2}\, V & 0\leq |\xi| \leq r_k , \\
 &  \\
0 & r_k < |\xi|  ,
\end{array}
\right.
\end{equation}
where $r_k$ is given by the condition
$$
(2\pi)^{-1}V \int_0^{r_k}\, |\xi|\, d|\xi| = k\, \,
\Rightarrow\, \, \, r_k = \sqrt{\frac{4\pi\, k}{V}}\, .
$$
Inserting \eqref{minimizer-li-yau} into \eqref{sum} we obtain
inequality \eqref{li-yau} for $d=2$. 

\subsection{Melas' improvement revisited}

Melas observed in \cite{Me} that the lower bound on the right-hand side of
\eqref{sum} can be improved, if one takes into account that the
follwing additional regularity condition on $F^*$ must hold
\begin{equation} \label{cond2}
 |(F^*)'|\, \leq 2(2\pi)^{-2}\, \sqrt{V\, I} \, =:L \, .
\end{equation}
It can be easily verified that, depending on the
value of $k$, the corresponding minimiser $\Phi_M$ of the functional
\eqref{sum} then has the following form:
\begin{equation} \label{minimizer-melas}
\text{for}\ k\geq
\frac{V^2}{48\pi I}\, \qquad \Phi_M(|\xi|) = \left\{
\begin{array}{l@{\quad \mathrm{} \quad }l}
(2\pi)^{-2}\, V & 0\leq |\xi| \leq s_k , \\
(2\pi)^{-2}\, V -(|\xi|-s_k)\, L & s_k < |\xi| \leq t_k, \\
0 & t_k < |\xi|,
\end{array}
\right.
\end{equation}
where the points $s_k$ and $t_k$ are uniquely determined by
$$
2\pi \int_{\R_+}\, \Phi_M(|\xi|)\, |\xi|\, d|\xi| = k, \quad t_k=
s_k+\frac{ V}{4\pi^2 L}\, ,
$$
\noindent see Figure 1, and 
\begin{equation} \label{minimizer-melas-b}
\text{for}\, k<\frac{V^2}{48\pi I} \qquad 
\Phi_M(|\xi|) = \left(\left(\frac{3kL^2}{\pi}\right)^{1/3}
  -L|\xi|\right)_+\, .
\end{equation}
Using this minimiser we obtain the lower bound
\begin{equation} \label{k-large}
\sum_{j=1}^k\, \lambda_j \geq \frac{2\pi}{V}\,
k^2+ \frac{1}{32}\, \frac{V}{I}\, k \qquad \text{if}\quad 
k\geq \frac{V^2}{48\pi I} 
\end{equation}
and
\begin{equation} \label{k-small}
\sum_{j=1}^k\, \lambda_j \geq \frac{2\pi}{V}\,
k^2+ \left(1-10\cdot 2^{-\frac 53} 3^{-\frac 43}\right) \frac{3}{10}\,
\left(\frac{2}{\pi}\right)^{\frac 23}\, L^{-\frac 23}\,  k^{\frac 53}
\qquad \text{if}\quad k < \frac{V^2}{48\pi I}.
\end{equation}
Now let $a\in\R^2$ be such that $I=\int_\Omega|x-a|^2\, dx$ and let
$B_a$ be the disc centred in $a$ and with the volume $V$. It is then
straightforward to verify that 
$$
I \geq I(B_a) = \frac{V^2}{2\pi}\, .
$$
Using this inequality and the fact that $k\geq 1$ we deduce from
\eqref{k-large} and \eqref{k-small} the uniform estimate
\begin{equation} \label{lin}
\sum_{j=1}^k\, \lambda_j \geq \frac{2\pi}{V}\,
k^2+ \frac{1}{32}\, \frac{V}{I}\, k \qquad \forall\, k
\in\N\, .
\end{equation}

\begin{figure}[t] \label{fig-min}
\begin{picture}(100,100)(-30,0)
\put(64,80){\makebox(0,0)[r]{$\frac{V}{4\pi^2}$}}
\put(64,115){\makebox(0,0)[r]{$F^*$}}

\put(68,60){\makebox(0,0)[r]{$\frac{V}{4\pi^2}-\eps k^{-\delta}$}}

\put(610,10){\makebox(0,0)[r]{$|\xi|$}}

\thicklines

\put(70,20){\vector(1,0){280}}

\put(70,20){\vector(0,1){100}}

\thinlines

\put(70,80){\line(1,0){120}}

\put(70,60){\line(1,0){170}}

\put(240,60){\line(0,-1){40}}

\put(190,80){\line(0,-1){60}}

\put(160,80){\line(1,-1){60}}

\put(227,10){\makebox(0,0)[r]{$t_k$}}

\put(248,10){\makebox(0,0)[r]{$\tau_k$}}

\put(223,40){\makebox(0,0)[r]{$\Phi_M$}}

\put(250,40){\makebox(0,0)[r]{$\Phi$}}

\put(195,10){\makebox(0,0)[r]{$r_k$}}

\put(190,90){\makebox(0,0)[r]{$\Phi_{LY}$}}
\end{picture}

\caption{Minimizers of the functional $\int_{\R^2}\, |\xi|^2\,
F^*(|\xi|)\, d\xi$.}
\end{figure}
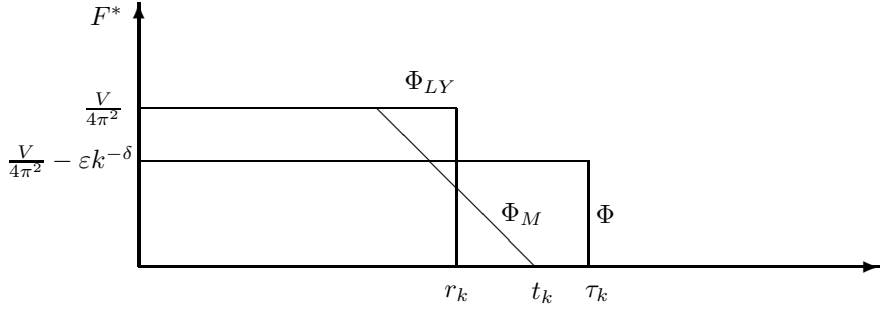

\vspace{0.2cm}

\subsection{The new correction term}
\label{correction}
Our main observation is that the crucial reservoir for improvements of
\eqref{li-yau} does not lie in the regularity of $F^*$, but in a more
detailed analysis and improvement of the condition  \eqref{maximum1}. 
Indeed, since 
\begin{equation} \label{newmax}
F(\xi) = \sum_{j=1}^k\, |f_j(\xi)|^2 = \frac{V}{4\pi^2}-
\sum_{j=k+1}^\infty\, |f_j(\xi)|^2,
\end{equation}
any estimate from below on $\sum_{j=k+1}^\infty\, |f_j(\xi)|^2$ will
automatically lead to a sharper upper bound on $F$ and therefore to
an additional term in the Li-Yau inequality. 

Moreover, the last term in \eqref{newmax} cannot go to zero
arbitrarily fast as $k$ goes to infinity. 
This follows from the fact that $|e^{-ix\cdot\xi}|=1$ everywhere in 
$\Omega$, which means that the Fourier coefficients
$f_j(\xi)$ of $e^{-ix\cdot\xi}$ with respect to the
basis $\{\psi_j\}$ cannot decay too fast in $j$ (each $\psi_j$
vanishes on $\pd\Omega$). In particular, the sequence $\{f_j(\xi)\}$
is not in $\ell^1$. Another way to see this is to realize
that the Fourier series $\sum_j\, f_j(\xi)\, \psi_j(\cdot)$ of continuous
functions approximates, in $L^2(\R^2)$, the function $
e^{-ix\cdot\xi}\chi_\Omega$, which has a discontinuity on $\pd\Omega$.  
Thus the decay properties of $\sum_{j=k+1}^\infty\, |f_j(\xi)|^2$ and
consequently the additional term in Li-Yau inequality should reflect
the effect of the boundary of $\Omega$. 

The main technical difficulty is to quantify this strategy into a
uniform lower bound on $\sum_{j=k+1}^\infty\, |f_j(\xi)|^2$. In
particular, if we can prove an estimate of the form
\begin{equation} \label{maximum2}
\sum_{j=k+1}^\infty\, |f_j(\xi)|^2 \, \geq \,\eps\, k^\delta 
\qquad \forall\xi\in\R^2,
\end{equation}
where  $\eps$ and $\delta$ are positive, then  
the corresponding minimiser of \eqref{sum} satisfying conditions 
\eqref{cond1} and \eqref{newmax} reads 
\begin{equation} \label{minimezer}
\Phi(|\xi|) = \left\{
\begin{array}{l@{\quad \mathrm{} \quad }l}
V/4\pi^2- \eps k^{-\delta} & 0\leq |\xi| \leq \tau_k , \\
0 & \tau_k < |\xi|  ,
\end{array}
\right.
\end{equation}
see Figure 1. Here $\tau_k$ is defined by the condition
$$
2\pi\int_{\R_+}\, \Phi(|\xi|)\, |\xi|\, d|\xi| = k.
$$
A direct calculation then shows that there exists a positive
coefficient $A(\eps, \delta)$ such that
\begin{equation} \label{demo}
\sum_{j=1}^k\, \lambda_j  \geq \, 2\pi\int_{\R_+}\Phi(|\xi|)\,
|\xi|^3\, d|\xi| = \frac{2\pi}{V}\, \, k^2 +
A(\eps,\delta)\, k^{2-\delta}.
\end{equation}
The asymptotic formula \eqref{weyl-second} implies that
$\delta \geq 1/2$. For $\delta <1$ we obtain an improvement of the
Melas bound.

\section{Main results}
\label{mainresult}
We will state and prove the results for the case of polygons and
general domains separately.

\subsection{Case 1: Polygons}

For a given polygon $\Omega$ we denote by $p_j,\, j=1,\dots ,n$ the
$j-$th side of $\Omega$. Moreover, we denote by $d_j$ the distance between the
middle third of $p_j$ to $\pd\Omega\setminus p_j$. We can now
formulate our first result.

\begin{Theorem}[Lower bound for polygons] \label{polygons}
Let $\Omega$ be a polygon with $n$ sides. Let $l_j$ be the length of
the $j-$th side of $\Omega$.
Then for any $k\in\N$ and any $\alpha\in[0,1]$ we have
\begin{equation} \label{polyg}
\sum_{j=1}^k\, \lambda_j \, \geq \, \frac{2\pi}{V}\, \, k^2
+4\alpha\, c_3\, k^{\frac 32-\eps(k)}\, \, V^{-\frac 32}\,
\sum_{j=1}^{n}\, l_j \, \Theta\left(k-\frac{9V}{2\pi\, d_j^2}\right)
+ (1-\alpha)\, \frac{V}{32\, I}\,  k,
\end{equation}
where 
\begin{equation} \label{eps}
\eps(k) = \frac{2}{\sqrt{\log_2(2\pi k/c_1)}}
\end{equation}
and
\begin{equation} \label{constants}
c_1 = \sqrt{\frac{3\pi}{14}}\ 10^{-11}\, , \qquad c_3 =
\frac{2^{-3}}{9\sqrt{2}\, 36}\, (2\pi)^{\frac 54} c_1^{1/4}\, .
\end{equation}
\end{Theorem}

\subsection{Case 2: General domains}

For general open domains  $\Omega\subset\R^2$ we will have to impose
certain assumptions on the regularity of $\pd\Omega$. 

\begin{ass} \label{Omega}
There
exist $C^2-$ smooth parts $\Gamma_j\subset\partial\Omega$ at the boundary
of $\Omega$. Let $j=1,\dots,m$.  
\end{ass}

\vspace{0.15cm}

\noindent To be able to state the result for general domains we need some
definitions.  Let $A_j,B_j$ be the end points of $\Gamma_j$ and let
$\{x^j_1(s),x^j_2(s)\}$ be the parametrisation of $\Gamma_j$ with its
length $s$.
We define
\begin{equation*}
\varkappa_j = \max_s\, |\varkappa_j(s)|,
\end{equation*}
where $\varkappa_j(s)$ denotes the curvature at the point
$s\in\Gamma_j$.
Moreover, let $L(\Gamma_j)$ be length of $\Gamma_j$.
Now we divide $\Gamma_j$ into
several pieces of the same length. The tiling of $\Gamma_j$ will be done
in two different ways depending on the values of $\varkappa_j$ and
$L(\Gamma_j)$: 

\begin{itemize}
\item[(i)] If
\begin{equation} \label{case1}
L(\Gamma_j) \leq \frac{3\pi}{8\, \varkappa_j},
\end{equation}
 then we
divide $\Gamma_j$ into three parts of the same length and denote by
$d_j$ the distance of the middle part to
$\partial\Omega_j\setminus\Gamma_j$.

\item[(ii)] If
\begin{equation} \label{case2}
L(\Gamma_j) > \frac{3\pi}{8\, \varkappa_j},
\end{equation}
then we divide $\Gamma_j$ into $n_j=\left[8
L(\Gamma_j)\varkappa_j/\pi\right]$ parts
of the same length. Let $a^j_i,a^j_{i+1}$ be the end points of the
$i-$th part with $a^j_0=A_j,\, a^j_{n_j}=B_j$ and let
$$
\delta_i^j = \text{dist}\left((a^j_i,a^j_{i+1}), \,
  \partial\Omega\setminus\{(a^j_{i-1},a^j_{i})\cup\, (a^j_i,a^j_{i+1})\,
  \cup\, (a^j_{i+1},a^j_{i+2})\}\right)
$$
Then we define
\begin{equation*}
d_j = \min_{1\leq i\leq n-2}\, \delta^j_i\, .
\end{equation*}
\end{itemize}

\noindent
Finally, we will need 
\begin{align*}\label{ki}
k_j &:=  \frac{V}{2\pi}\, \max\left\{\Lambda_3(j),\,
  \frac{9}{d^2_j},\, \frac{128\, \varkappa^2_j}{\pi^2}\, ,
\, \frac{6\varkappa_j}{d_j}\right\}\, ,
\end{align*}
where
$$
\Lambda_3(j) := \max\left\{9\cdot 2^{10} \max_j\, \varkappa_j^2 ,\, 2^{2^6}\, c_1
V^{-1}\, , \,
  c_1^{-1}\, 2^{22}\, 6^8\, \varkappa_j^4\, V\right\}\, .
$$

\noindent Now we are in position to state the result for general
domains.

\begin{Theorem}[Lower bound for general domains]  \label{maintheorem}
Let $\Omega$ satisfy Assumption \ref{Omega}.
Then for any $k\in\N$ and any  $\alpha\in[0,1]$ we have
\begin{equation} \label{main}
\sum_{j=1}^k\, \lambda_j  \, \geq \, \frac{2\pi}{V}\, \, k^2
+\alpha\, c_3\, k^{\frac 32-\eps(k)}\,
 V^{-3/2}\, \sum_{j=1}^{m}\, L(\Gamma_j)\, \Theta(k-k_j)
+ (1-\alpha)\, \frac{V}{32\, I}\,  k.
\end{equation}
\end{Theorem}

\subsection{Remarks}

\begin{Remark}
Note that the coefficient of the second term on the right hand side
of \eqref{main} is very similar to the coefficient of the second term 
in the Weyl asymptotics
\eqref{weyl-second}. In particular, it reflects the expected effect of
the boundary of $\Omega$. On the other hand, this boundary term
becomes visible only for $k$ large enough. However, we would like to
point out that the second term cannot be simply proportional to $\sum_j\,
L(\Gamma_j)$. Indeed, one can make  $\sum_j\,
L(\Gamma_j)$ arbitrarily large by ``folding'' the boundary $\pd\Omega$
while
keeping the eigenvalues $\lambda_j$ with $j\leq k$ 
almost unchanged. This shows that
the condition $k\geq k_j$ cannot be removed.  
\end{Remark}

\begin{Remark}
It would be natural to try to deduce the result for general domains
from the result for polygons by approximating $\Omega$ by polygons. 
However, the contribution of the second term would in general
disappear in such a procedure. To see this it suffices to
take an open ball in $\R^2$ as $\Omega$. 
Then the coefficients $k_j$ would go to infinity 
when approximating $\Omega$ by a sequence of polygons.    
Therefore a different strategy will be needed in the proof of Theorem
\ref{maintheorem}. 
\end{Remark}

\begin{Remark}
As for the constants in \eqref{main}, notice that $\eps(k)\ll 1$ for
all $k$ and
that $\eps(k)\to 0$ as $k\to\infty$. On the other hand, the values of
$k_j$ are in general very large. Nevertheless, the correction term on
the right-hand side of \eqref{main} can
be optimised according to the geometry of $\Omega$ by choosing the
boundary segments $\Gamma_j$ in an appropriate way. 
\end{Remark}


\section{Proof for polygons}
\label{polygons-proof}

The proofs of our main results rely on a careful
exploitation of the ideas described in section \ref{correction}.   

Let $\lambda=\lambda_k$ and let 
$
\el_k:= \left\{\sum_{i=1}^k c_i\psi_i: \, \sum_{i=1}^k\, |c_i|^2 \leq
  V\right\}
$.
Since $e^{i\xi\cdot x}$ belongs to $L^2(\Omega)$ for each
$\xi\in\R^2$, it follows that 
\begin{align} \label{infimum}
\inf_{\psi\in\el_k} \, \left\|e^{i\xi\cdot
    x}-\psi\right\|_{L^2(\Omega)}^2  \, \leq\,
\|e^{i\xi\cdot x}-\sum_{i=1}^k \left(e^{i\xi\cdot x},\,
\psi_i\right)_{L^2(\Omega)}\, \psi_i\|_{L^2(\Omega)}^2 
 = V -4\pi^2\, F(\xi)\, ,
\end{align}
where
$$
\sum_{i=1}^k \left|\left(e^{i\xi\cdot x},\,
\psi_i\right)_{L^2(\Omega)}\right|^2 = 4\pi^2 F(\xi) \leq V .
$$
Equation \eqref{infimum} yields the estimate
$$
F(\xi) \leq (4\pi^2)^{-1}\, \left(V- \inf_{\psi\in\el_k} \,
  \left\|e^{i\xi\cdot
    x}-\psi\right\|_{L^2(\Omega)}^2\right) .
$$
In view of the arguments given in section \ref{correction}, to prove
\eqref{polyg} it thus suffices to show that
\begin{equation} \label{estimate-lambda}
\left\|e^{i\xi\cdot x}-\psi\right\|_{L^2(\Omega)}^2 \, \geq \,
\text{const\,} k^{-\frac 12-\eps(k)}\qquad \forall\psi\in\el_k\,
\end{equation}
holds for $k$ large enough. Moreover, it is well known that $\lambda_k
\sim k$ in dimension $d=2$, which shows that \eqref{estimate-lambda} is
equivalent to
\begin{equation} \label{estimate-lambda-a}
\left\|e^{i\xi\cdot x}-\psi\right\|_{L^2(\Omega)}^2 \, \geq \,
\text{const\,} \lambda_k^{-\frac 12-\eps(k)}\qquad \forall\psi\in\el_k\, .
\end{equation}
The idea how to prove \eqref{estimate-lambda-a} is obvious; since
$|e^{i\xi\cdot x}|=1$ everywhere and
$\psi=0$ on $\partial\Omega$, we will estimate the left-hand
side of \eqref{estimate-lambda} by integrating 
 over a
suitable neighbourhood of $\pd\Omega$ only. More precisely, we will
make use of the contributions from integrating $|e^{i\xi\cdot
x}-\psi|^2$ over squares of the size of order
$\lambda^{-1/2}$ attached to the boundary of $\Omega$, see Figure 2.   
To estimate these contributions from below, we will 
need appropriate integral upper bounds on the normal derivatives of $\psi$ on
$\partial\Omega$ in terms of $\lambda$. This will be done as the first
step of the proof. 


\subsection{Eigenfunctions estimates}
In this section we give an $L^2$ estimate on the derivatives the
eigenfunctions $\psi_i$ in the vicinity of $\pd\Omega$. Let
$$
\omega=\left[0,\frac{1}{2\sqrt{\lambda}}\right]\times\left[-\frac{1}{4\sqrt{
\lambda}}, \frac{1}{4\sqrt{\lambda}}\right] 
$$
and assume that $\lambda$ is large enough so that the square $\omega$
can be placed inside $\Omega$ in such a way that one of its sides
coincides with a part of $\pd\Omega$, see Figure 2. 
We also introduce a local system of
coordinates $(x_1,x_2)$  as in Figure 2. Finally, for a given
$p\in\N$ we define the sequence $A_n(p)$ by
\begin{equation} \label{a_n}
A_n(p)= (3+726\cdot 4^6\, p^4)\,  A_{n-2}(p) + 150\cdot 9^2\, p^2\,
A_{n-1}(p)
\end{equation}
where $A_0(p)=1$ and $A_1(p)=1$. We then have

\begin{figure}[h]
\begin{picture}(110,110)(-30,0)
\put(90,80){\vector(1,0){180}}
\put(90,80){\circle*{4}}
\put(90,90){\makebox(0,0)[r]{$t_l$}}
\put(120,90){\makebox(0,0)[r]{$\pd\Omega$}}
\put(243,80){\circle*{4}}
\put(250,90){\makebox(0,0)[r]{$t_{l+1}$}}
\put(130,33){\makebox(0,0)[r]{$\Omega$}}
\put(270,90){\makebox(0,0)[r]{$x_2$}}
\linethickness{1.4pt}
\put(138,80){\line(1,0){60}}
\put(138,80){\line(0,-1){60}}
\put(138,20){\line(1,0){60}}
\put(198,20){\line(0,1){60}}
\thinlines
\put(167,80){\vector(0,-1){80}}
\put(167,-5){\makebox(0,0)[r]{$x_1$}}
\put(185,28){\makebox(0,0)[r]{$\omega$}}
\end{picture}
\caption{Construction of the local coordinate system at the boundary
  of $\Omega$. The end points of the $l-$th side of $\Omega$ are
  denoted by $t_l$ and $t_{l+1}$ respectively.}
\end{figure}
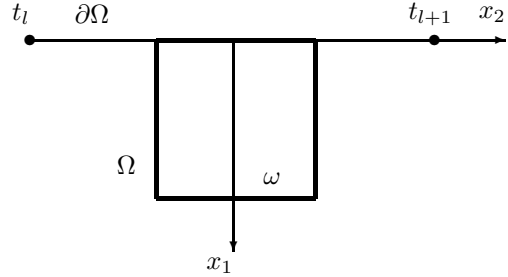

\begin{Lemma} \label{rect}
Let $\psi_i$ be a normalised eigenfunction of the Dirichlet
Laplacian on $\Omega$ with an eigenvalue $\lambda_i \leq \lambda$.
Then
\begin{equation} \label{derivative}
\left\|\frac{\pd^{p+1}\psi_i}{\pd x_1^{p+1}}\right\|^2_{L^2(\omega)}\,
\leq \, A_p(p)\, \lambda^{p+1}
\end{equation}
holds true for all $p\in\N_0$.
\end{Lemma}

\begin{proof}
For $n,p\in\N$
we define the functions $g:[0,1]\to [0,1]$ by \newline
$g(x):=1-6x^4+8x^6-3x^8$ and $v_{n,p}:\R\to\R$ by
\begin{equation*}
v_{n,p}(t) = \left\{
\begin{array}{l@{\quad \mathrm{} \quad }l}
1  & 0\leq t \leq \frac{2p-n}{2p} , \\
& \\
g(2pt-2p+n) &  \frac{2p-n}{2p}\, \leq t \leq \frac{2p-n+1}{2p}, \\
& \\
0 & \frac{2p-n+1}{2p} < t
\end{array}
\right.
\end{equation*}
with $v_{n,p}(t)=v_{n,p}(-t)$ for $t<0$. It is easy to check that
\begin{equation*}
|v_{n,p}(t)|\leq 1,\quad |v'_{n,p}(t)| \leq 2\alpha_1\, p,\quad
|v''_{n,p}(t)| \leq 4\alpha_2\, p^2\, ,
\end{equation*}
where $\alpha\leq 5/2$ and $\alpha_2\leq 11$. Next we define
$$
W_{n,p,\lambda}(x_1,x_2)= v_{n,p}(\sqrt{\lambda}\, x_1)\,
v_{n,p}(4\sqrt{\lambda}\, x_2)
$$
and note that
\begin{equation}
\begin{aligned} \label{bounds}
|W_{n,p,\lambda}(x_1,x_2)| & \leq 1, \quad |\nabla
W_{n,p,\lambda}(x_1,x_2)| \leq 9 \sqrt{\lambda}\, \alpha_1\, p
\\
|\Delta W_{n,p,\lambda}(x_1,x_2)| & \leq \sqrt{2}\, 4^3\, \lambda\,
\alpha_2 p^2
\end{aligned}
\end{equation}
for all $(x_1,x_2)\in\omega$. We will prove
\begin{equation} \label{assumption}
\begin{aligned}
\left\|\frac{\pd^n\psi_i}{\pd x_1^n}\right\|^2_{L^2(supp\, W_{n-1,p,\lambda})}\,
& \leq \, A_{n-1}(p)\, \lambda^{n}, \\
\left\|\frac{\pd^{n}\psi_i}{\pd x_1^{n-1}\pd x_2}
\right\|^2_{L^2(supp\, W_{n-1,p,\lambda})}\,
&\leq \, A_{n-1}(p)\, \lambda^{n}
\end{aligned}
\end{equation}
by induction in $n$ for $n=1,\dots,p$. Notice that, in view of
\eqref{lambda-critical-1}, \eqref{lambda-critical-2}, the inclusion
$$
\omega_{n} := (supp\, W_{n,p,\lambda}) \subset \Omega
$$
holds true for every $p\in\N$ and every $n\leq p$. For $n=1$ we have
$$
\left\|\frac{\pd\psi_i}{\pd x_1}\right\|^2_{L^2(\omega_0)}\leq
A_0(p)\, \lambda , \qquad
\left\|\frac{\pd\psi_i}{\pd x_2}\right\|^2_{L^2(\omega_0)}\leq
A_0(p)\, \lambda.
$$
Multiplying the equation $-\Delta\psi_i =\lambda_i\psi_i$ by
$\frac{\pd^2\psi_i}{\pd x_1^2}$ and integrating by parts we find out
that
$$
\left\|\frac{\pd^2\psi_i}{\pd x_1^2}\right\|^2_{L^2(\omega_1)}\leq
A_1(p)\, \lambda^2 , \qquad
\left\|\frac{\pd^2\psi_i}{\pd x_1\pd x_2}\right\|^2_{L^2(\omega_1)}\leq
A_1(p)\, \lambda^2 \, .
$$
Hence \eqref{assumption} holds for $n=1$ and $n=2$.
Now assume that \eqref{assumption} holds for some $n-1$ and $n$. We will
show that it holds for $n+1$ as well. Integration by parts yields
\begin{align} \label{inter}
& \left\|\Delta\left(\frac{\pd^{n-1}\psi_i}{\pd x_1^{n-1}}
W_{n-1,p,\lambda}\right)\right\|^2_{L^2(\omega_{n-1})}\,  =
\left\|\frac{\pd^2}{\pd x_1^2}\left(\frac{\pd^{n-1}\psi_i}{\pd x_1^{n-1}}
W_{n-1,p,\lambda}\right)\right\|^2_{L^2(\omega_{n-1})}\, + \nonumber \\
& \left\|\frac{\pd^2}{\pd x_2^2}\left(\frac{\pd^{n-1}\psi_i}{\pd
      x_1^{n-1}}
W_{n-1,p,\lambda}\right)\right\|^2_{L^2(\omega_{n-1})}\, +
2\, \left\|\frac{\pd^2}{\pd x_1\pd x_2}\left(\frac{\pd^{n-1}\psi_i}{\pd
      x_1^{n-1}}\,
W_{n-1,p,\lambda}\right)\right\|^2_{L^2(\omega_{n-1})}
\end{align}
From the fact that $W_{n-1,p,\lambda}=1$ on the $\omega_{n}$
it follows that the first
and the last term on the right hand side of \eqref{inter} are
greater than or equal to
$$
\left\|\frac{\pd^{n+1}\psi_i}{\pd x_1^{n+1}}\right\|^2_{L^2(\omega_{n})}\,
\,  \text{and} \qquad
\left\|\frac{\pd^{n+1}\psi_i}{\pd x_1^{n}\pd
    x_2}\right\|^2_{L^2(\omega_{n})}
$$
respectively. The second term on the right hand side of \eqref{inter}
is positive and since $\omega\subset supp\,  W_{n,p,\lambda}$, we get
\begin{align}
\left\|\frac{\pd^{n+1}\psi_i}{\pd x_1^{n+1}}\right\|^2_{L^2(\omega_n)}\, +
\left\|\frac{\pd^{n+1}\psi_i}{\pd x_1^{n}\pd x_2}\right\|^2_{L^2(\omega_n)}
\leq \left\|\Delta\left(\frac{\pd^{n-1}\psi_i}{\pd x_1^{n-1}}
W_{n-1,p,\lambda}\right)\right\|^2_{L^2(\omega_n)}\, .
\end{align}
Next we employ \eqref{bounds} and \eqref{assumption} to conclude that
\begin{align}
&\left\|\Delta\left(\frac{\pd^{n-1}\psi_i}{\pd x_1^{n-1}}
W_{n-1,p,\lambda}\right)\right\|^2_{L^2(\omega_n)}\, =  \\
& \quad \left\|\lambda_i \left(\frac{\pd^{n-1}\psi_i}{\pd
x_1^{n-1}}\right) W_{n-1,p,\lambda} +
\left(\frac{\pd^{n-1}\psi_i}{\pd x_1^{n-1}}\right) \Delta
W_{n-1,p,\lambda}+2\left(\nabla\frac{\pd^{n-1}\psi_i}{\pd
x_1^{n-1}}\right)
\nabla W_{n-1,p,\lambda}\right\|^2_{L^2(\omega_n)} \nonumber \\
& \nonumber \\
& \quad \leq 3\lambda^{n+1} A_{n-2}(p) + 6\cdot 4^6\, \alpha_2^2\, p^4\,
\lambda^{n+1}\, A_{n-2}(p)
+ 24\cdot9^2\, \alpha_1^2\, p^2\, \lambda^{n+1}\, A_{n-1}(p) \, 
\leq \lambda^{n+1}\, A_n(p) \, . \nonumber
\end{align}
\end{proof}

\noindent As a consequence of this result we obtain

\begin{Corollary} \label{corol-rect}
Let $\omega$ be as in Lemma \ref{rect}. Assume that
$\psi=\sum_{\lambda_i\leq \lambda} c_i\psi_i$ with $\sum_{\lambda_i\leq
  \lambda} |c_i|^2\leq V$. Then
\begin{equation*}
\left\|\frac{\pd^{p+1}\psi}{\pd x_1^{p+1}}\right\|^2_{L^2(\omega)}\,
\leq \, \frac{A_p(p) V^2(\Omega)}{4\pi}\, \lambda^{p+2} .
\end{equation*}
\end{Corollary}

\begin{proof}
By Lemma \ref{rect} and the Cauchy-Schwarz inequality we have
\begin{align}
\left\|\frac{\pd^{p+1}\psi}{\pd x_1^{p+1}}\right\|^2_{L^2(\omega)}\,
\leq \sum_{\lambda_i\leq\lambda} |c_i|^2 \sum_{\lambda_i\leq\lambda}
\left\|\frac{\pd^{p+1}\psi_i}{\pd x_1^{p+1}}\right\|^2_{L^2(\omega)}
\, \leq \, N_\lambda\, V\, A_p(p)\, \lambda^{p+1}\, ,
\end{align}
Using the lower bound on $\lambda_i$ given in \eqref{k-estimate} we find
out that
$
N_\lambda \, \leq
 \frac{V}{4\pi}\, \lambda.
$
\end{proof}


\subsection{Lower bound on a square}

Corollary \ref {corol-rect} is one the two main technical
results on which is based the proof of Theorems \ref{polygons} and
\ref{maintheorem}. The goal of this section is to prove the second one
of these results, namely Proposition \ref{prop2} (see page \pageref{prop2}).  
We start with a couple of one dimensional estimates concerning smooth
functions on an interval $[0,l]$. Unless otherwise stated,
$\|\cdot\|$ denotes the $L^2-$norm on $[0,l]$.

\begin{Lemma} \label{max}
Let $f\in C^{p+1}[0,l],\, p\in\N$. Then
\begin{equation*}
\max\, |f^{(p)}|^2 \, \leq \, \frac 32\, \left(\frac
  1l\, \|f^{(p)}\|^2+l\, \|f^{(p+1)}\|^2\right).
\end{equation*}
\end{Lemma}

\begin{proof}
Let $\max\, |f^{(p)}| = |f^{(p)}(t_0)|$ with
$t_0\in[0,l]$. For any $t\in[0,l]$ we have
$$
f^{(p)}(t)=f^{(p)}(t_0)+\int_{t_0}^t\, f^{(p+1)}(\tau)\, d\tau\, .
$$
Integrating with respect to $t$ and using the Jensen inequality gives 
\begin{align}
l\, |f^{(p)}(t_0)|^2  & \leq \frac 32\, \int_0^l|f^{(p)}(t)|^2\,
dt+ 3\int_0^l\, \left(\int_{t_0}^t\, f^{(p+1)}(\tau)\,
  d\tau\right)^2\, dt
\nonumber \\
& \leq \frac 32\, \|f^{(p)}\|^2+3 \int_0^l\, t \|f^{(p+1)}\|^2\, dt =
\frac 32 \left(\|f^{(p)}\|^2+l^2 \|f^{(p+1)}\|^2\right). \nonumber
\end{align}

\end{proof}

\begin{Lemma} \label{22}
Let $f\in C^2\left[0,\frac 12\, \lambda^{-1/2}\right]$ and
real-valued. Then one of the following inequalities holds true:
\begin{equation} \label{ineq1}
 \max |f|\, \max |f''|\, \leq\, \frac 14\, \max |f'|^2 
\end{equation}
\begin{equation} \label{ineq2}
\qquad \qquad \ \ \max |f'|\, \leq\,
32\, \lambda^{\frac 12}\, \max |f|
\end{equation}

\end{Lemma}

\begin{proof} Let $m_i = \max\, |f^{(i)}|,\,
i\in\{0,1,2\}$ and let $t_0\in \left[0,\frac 12\,
\lambda^{-1/2}\right]$ be such that $m_1=|f'(t_0)|$. Without loss of
generality we assume that $t_0 < \frac 14\, \lambda^{-1/2}$, otherwise
we consider the interval $[0,t_0]$ instead of $[t_0, \frac 12\,
\lambda^{-1/2}]$. Assume that $f'(t_0) =m_1$. If
\begin{equation}
t_0 + \frac{m_1}{m_2} \, \leq \,  \frac 12\, \lambda^{-1/2} ,
\end{equation}
then the Taylor theorem says that
$$
m_0 \geq f\left(t_0+\frac{m_1}{m_2}\right) \geq f(t_0)+m_1
\left(\frac{m_1}{m_2}\right) -\frac{m_2}{2}\,
\left(\frac{m_1}{m_2}\right)^2 \geq -m_0+\frac{m_1^2}{2m_2}\, ,
$$
which implies \eqref{ineq1}. If, on the contrary,
\begin{equation*}
t_0 + \frac{m_1}{m_2} \, > \,  \frac 12\, \lambda^{-1/2} , \qquad
\text{then}\qquad 
\frac{m_1}{m_2} \, > \,  \frac 12\, \lambda^{-1/2}-t_0 > \frac 14\,
\lambda^{-1/2} . 
\end{equation*}
In this case we have
$$
m_0 \geq f\left(t_0+\frac 18\, \lambda^{-1/2}\right) \geq f(t_0)+m_1\,
\frac 18\, \lambda^{-1/2} -\frac{m_2}{128}\, \lambda^{-1}\, ,
$$
which implies
$$
m_1\, \frac 18\, \lambda^{-1/2} -\frac{m_1}{32\lambda^{-1/2}}\,
\lambda^{-1}\, \leq 2m_0.
$$
From here we conclude that
$$
\frac{m_1}{m_0} \leq \frac{64}{3\lambda^{-1/2}}\, \leq 32\,
\lambda^{\frac 12}\, .
$$
The proof in the case $f'(t_0)=-m_1$ is analogous.
\end{proof}

\begin{Proposition} \label{prop1}
Let $f\in C^p\left[0,\frac 12\, \lambda^{-1/2}\right],\, p\in\N$ and
let $f$ be
real-valued. Then one of the following inequalities holds true:
\begin{equation} \label{ineq3}
\max|f'|\, \leq \, 4^{p+\frac 12}\, \lambda^{\frac 12}\, \max|f|
\end{equation}
\begin{equation} \label{ineq4}
\qquad \qquad \quad \ \ \max|f'| \, \leq\, \left(\frac{\max
    |f^{(p)}|}{\max|f|}\right)^{\frac 1p}\,  4^{p-\frac 12}\, \max
|f|\, .
\end{equation}
\end{Proposition}

\begin{proof} Let $m_i= \max|f^{(i)}|,\, i=1,\dots,p$. There are two
possibilities. Either for all $i \leq p$ holds
\begin{equation} \label{case-a}
\frac{m_i}{m_{i-1}}\, \geq 32\, \lambda^{\frac 12}\, ,
\end{equation}
or there exists $i_0\in[1,p]$ , such that
\begin{equation} \label{case-b}
\forall\, i< i_0\quad \frac{m_i}{m_{i-1}}\, \geq 32\, \lambda^{\frac 12}\, ,
\qquad \frac{m_{i_0}}{m_{i_0-1}}\, < 32\, \lambda^{\frac 12}\, .
\end{equation}
In the first case $\frac{m_i}{m_{i-1}}\,
> \frac 14\, \frac{m_{i-1}}{m_{i-2}}$ holds for all $i\leq p$, see Lemma
\ref{22}. This yields
$$
m_p \geq 4^{-\frac{p(p-1)}{2}}\, \left(\frac{m_1}{m_0}\right)^p\, m_0,
$$
which is equivalent to \eqref{ineq4}. In the second case we have
$\frac{m_i}{m_{i-1}}\, > \frac 14\, \frac{m_{i-1}}{m_{i-2}}$  for all
$i\leq i_0$. Combining this with \eqref{case-b} we conclude that
$$
\frac{m_1}{m_0}\, \leq \, 4^{i_0+\frac 12}\, \lambda^{\frac 12}\, .
$$
\end{proof}

\begin{Corollary} \label{cor1}
Let $f\in  C^p\left[0,\frac 12\, \lambda^{-1/2}\right],\, p\in\N$ be a
complex-valued function such that $f(0)=0$ and $\max|f^{(p)}| \leq
C(p)\, \lambda^{\frac p2+1}$ for some constant $C(p)$. Then for any
$\varphi_0,\varphi_1\in\R$ holds
\begin{equation} \label{lower-bound1}
\int_0^{\frac 12 \lambda^{-1/2}}\, |f(t)-e^{i\varphi_1
t+i\varphi_0}|^2\, dt  \geq \frac{\lambda^{-\frac 12}}{9}
\, \min\left\{ 4^{-p-\frac 52}\,\, ,\, 4^{-\frac{p+3}{2}}\, 6^{\frac
    1p}\, C(p)^{-\frac 1p}\,
\lambda^{-\frac 1p}\right\}.
\end{equation}

\end{Corollary}

\begin{proof}
Let $u= \text{Re} f$ and $v= \text{Im} f$. If $\max|f|\geq 6$, then at
least one the expressions $\max|u|,\, \max|v|$ is larger than or equal
to $3$. Without loss of generality we assume that $\max|u| \geq 3$ and
apply Proposition \ref{prop1} to the function $u$.
If $u$ satisfies \eqref{ineq3}, then there exists an subinterval
$I\subset [0,\frac 12\, \lambda^{-1/2}]$ of the
length $3^{-1}4^{-p-\frac 12}\, \lambda^{-\frac 12}$ on which $|u| \geq
3/2$. This implies
$$
\int_0^{\frac 12 \lambda^{-1/2}}\, |f(t)-e^{i\varphi_1
t+i\varphi_0}|^2\, dt  \geq 3^{-1}\, 4^{-p-\frac 12}\, \lambda^{-\frac
12}\, .
$$
If, on the other hand, $u$ satisfies \eqref{ineq4}, then the length of
the subinterval of $[0,\frac 12\, \lambda^{-1/2}]$,
on which $|u| \geq 3/2$, is at least
$3^{-1}4^{-\frac{p-1}{2}}\, C(p)^{-\frac 1p}\, \lambda^{-\frac 12-\frac
  1p}$, which gives
$$
\int_0^{\frac 12 \lambda^{-1/2}}\, |f(t)-e^{i\varphi_1
t+i\varphi_0}|^2\, dt  \geq 3^{-1}\, 4^{-\frac{p-1}{2}}\, C(p)^{-\frac
1p}\, \lambda^{-\frac 12-\frac 1p} .
$$
Assume now that $\max|f| <6$. The latter means that $\max|u|< 6$ and
$\max|v|<6$. Since $u(0)=v(0)=0$,  there exists a subinterval of
$[0,\frac 12\, \lambda^{-1/2}]$, on which
$\max\{|u(t)|,|v(t)|\} \leq 1/3$, which implies $|f(t)-e^{i\varphi_1
t+i\varphi_0}|^2 \geq 1/4$. Applying Proposition \ref{prop1} to the
functions $u,v$ we find out that the length of this interval is
bounded from below by
$$
\min\left\{3^{-2}\, 4^{-p-\frac 52}\, \lambda^{-\frac 12}\, ,\,
3^{-2}\, 4^{-\frac{p+3}{2}}\, 6^{\frac
    1p}\, C(p)^{-\frac 1p}\, \lambda^{-\frac
  12-\frac 1p}\right\}\, .
$$
This completes the proof.
\end{proof}

\noindent With the above auxiliary results at hand, we can finally
prove the following integral estimate, which will play a
central role in the proof of Theorem \ref{polygons} and \ref{maintheorem}.

\begin{Proposition} \label{prop2}
Let $f\in C^{p+1}[\omega]$ a complex valued function such that
$f(0,x_2) = 0$ for each $x_2$ and
\begin{equation*}
\left\|\frac{\pd^{p+1} f}{\pd x_1^{p+1}}\right\|_{L^2(\omega)}
\leq\,
\beta_{p+1}\, \lambda^{1 +\frac p2}, \quad
\left\|\frac{\pd^{p} f}{\pd x_1^{p}}\right\|_{L^2(\omega)} \leq\,
\beta_{p}\, \lambda^{\frac 12 +\frac p2}
\end{equation*}
for some positive $\beta_p$ and $\beta_{p+1}$. Then the
inequality
\begin{align}
\left\|f- e^{i(\xi_1x_1+\xi_2x_2+\varphi)}\right\|^2_{L^2(\omega)}
\geq \frac{1}{36}\,\min\left\{4^{-p-\frac 52}\, \lambda^{-1},\,
  4^{-\frac p2-\frac 32}\, 6^{\frac{1}{2p}}\,
  (\beta^2_{p+1}+\beta^2_p)^{-\frac{1}{2p}}\, \lambda^{-1-\frac 1p}\right\}
\end{align}
holds true for all $\xi_1,\xi_2,\varphi\in\R$.
\end{Proposition}

\begin{proof}
The measure of the set
$$
\left\{x_2\in[0,\lambda^{-1/2}]\, :\,
  \int_0^{\frac{1}{2\sqrt{\lambda}}}\left|\frac{\pd^{i} f(x_1,x_2)}{\pd
      x_1^{i}}\right|^2\, dx_1 \leq\, 8\, \beta_i^2 \lambda^{i+\frac 32}, \,
  i\in\{p,p+1\}\right\}
$$
is obviously at least $\frac 14\, \lambda^{-\frac 12}$. For such $x_2$
holds by Lemma \ref{max}
$$
\max_{x_1}\left|\frac{\pd^{p} f(x_1,x_2)}{\pd x_1^{p}}\right| \leq
\sqrt{3}\, \, \lambda^{1+\frac p2}\, \sqrt{\beta_{p+1}^2+\beta_p^2}\, .
$$
Corollary \ref{cor1} then implies the statement.
\end{proof}

\subsection{Proof of Theorem \ref{polygons}}

\begin{proof}[Proof of Theorem \ref{polygons}]
Fix $\lambda >0$. 
Let $\lambda_j$ be the eigenvalues of the Dirichlet Laplacian on $\Omega$
and let $\psi_j$ be the corresponding normalised
eigenfunctions. For $k\in\N$ we define
\begin{equation*}
F(\xi) = \sum_{j=1}^k\, |\hat{\psi}_j(\xi)|^2 \, ,
\end{equation*}
where $\hat{\psi_j}$ denotes the Fourier transform of $\psi_j$.
Moreover, we denote by $F^*(|\xi|)$ the decreasing radial
rearrangement of $F(\xi)$. Let
$$
\psi(x)= \sum_{\lambda_i \leq \lambda}\, c_i\, \psi_i(x),\quad
\text{with\, \, }  
\sum_{\mu_i \leq\lambda}|c_i|^2 \leq V \, .
$$

\noindent For each $j=1,\dots ,n$ we choose on the middle part of $p_j$
several points $t_l$ such that dist$(t_l,t_{l+1})=\sqrt{2}\,
\lambda^{-1/2}$ for all $l$ and denote
by  $T_l$ the squares with the side
$\frac 12 \lambda^{-1/2}$ constructed in the middle point between $t_l$ and
$t_{l+1}$, see Figure 2. We note that for each $j$ the number of these
squares is at least
\begin{equation*} 
N_j = \left [\frac{1}{3\sqrt{2}}\ l_j\, \lambda^{\frac
  12}\right]\, .
\end{equation*}
According to Corollary \ref{corol-rect} for each $l$ and $p$ we have
\begin{equation*}
\left\|\frac{\pd^{p+1}\psi}{\pd \nu^{p+1}}\right\|^2_{L^2(T_l)}\,
\leq \, \frac{A_p(p) V^2}{4\pi}\, \lambda^{p+2} ,
\end{equation*}
where $\frac{\pd\psi}{\pd \nu}$ denotes the normal derivative of
$\psi$.
In view of Proposition \ref{prop2} and Corollary \ref{V-extended}
we get
\begin{equation} \label{vol}
\begin{aligned}
\left\|\psi-e^{i\xi\cdot x} \right\|^2_{L^2(T_l)}\,
 \geq \frac{1}{36}\,\min\left\{4^{-p-\frac 52}\, \lambda^{-1},\,
  4^{-\frac p2-\frac 32}\, 6^{\frac{1}{2p}}\,
  (\beta^2_{p+1}+\beta^2_p)^{-\frac{1}{2p}}\, \lambda^{-1-\frac 1p}\right\},
\end{aligned}
\end{equation}
where
$$
\beta^2_{p+1} = \frac{A_p(p) V^2}{4\pi}\, .
$$
We continue by estimating the sequence $A_p(p)$. A direct inspection
shows that
\begin{equation} \label{A_p}
A_p(p) \leq  c_0\, 2^{(p+1)^2}\, , \quad c_0= 7\cdot 10^{22}\, .
\end{equation}
This implies that $(\beta^2_{p+1}+\beta^2_p)^{-\frac{1}{2}}\geq
  2^{-1/2}\, \sqrt{\pi}\, c_0^{-1/2}\, V^{-1}\, 2^{-(p+1)^2/2}$. Hence for
\begin{equation*}
p = \left[\sqrt{2\log_2 (V\lambda/c_1)}\right] -1,\quad c_1=
\sqrt{\frac{3\pi}{2}}\, c_0^{-\frac 12}
\end{equation*}
we obtain
\begin{equation*}
\left\|\psi-e^{i\xi\cdot x} \right\|^2_{L^2(T_l)}\, \geq
\frac{2^{-3}}{36}\, c_1^{-1}\,
  V\left(\frac{V\lambda}{c_1}\right)^{-1-\frac{2}{\sqrt{\log_2
          (V\lambda/c_1)}}}\, .
\end{equation*}
Taking $\lambda$ large enough such that 
\begin{equation*}
\lambda^{-1/2}\, \leq \, \frac{d_j}{3}\, .
\end{equation*}
we make sure that the squares $T_l$ lie inside $\Omega$ and that they 
do not overlap each other.
Summing this inequality for all $l=1,\dots ,N_j$ and all $j=1,...,n$
we thus arrive at
\begin{align}
V -4\pi^2\, F^*(|\xi|)  \geq
\left\|\psi-e^{i\xi\cdot x} \right\|^2_{L^2(\Omega^e)}\,
\geq c_2\, V^{\frac 12}\,
  \left(\frac{V\lambda}{c_1}\right)^{-\frac 12-\frac{2}{\sqrt{\log_2
(V\lambda/c_1)}}}\, \sum_{j=1}^{n}\,
l_j\, \Theta\left(\lambda-\frac{9}{d_j^2}\right) 
\end{align}
with
$c_2 = \frac{2^{-3}}{9\sqrt{2}\, 36}\, c_1^{-1/2}$. This yields the
following upper bound on $F^*$:
\begin{align} 
F^*(|\xi|)\leq  \ M(p,\lambda) 
:= \frac{V}{4\pi^{2}}\,
\left[1-c_2\, V^{-\frac 12}\,
  \left(\frac{V\lambda}{c_1}\right)^{-\frac
    12-\frac{2}{\sqrt{\log_2(V\lambda/c_1)}}}\,
\, \sum_{j=1}^{n}\, l_j\, \Theta\left(\lambda-\frac{9}{d_j^2}\right) 
\right] . 
\end{align}
Now we use the minimiser
\eqref{minimizer-li-yau} with $V/4\pi^2$ replaced by
$M(p,\lambda)$ to obtain
\begin{equation} \label{second}
\sum_{j=1}^k\, \lambda_j \, \geq\, 
\int_{\R^2}F^*(|\xi|)|\xi|^2\, d\xi\, \geq \frac{\lambda^2 V^2}{8\pi^3\,
  M(p,\lambda)}\, .
\end{equation}
Employing the definition of $M(p,\lambda)$ we then find out that
\begin{align} \label{pre-ineq}
\sum_{j=1}^k\, \lambda_j & \geq  \frac{\lambda^2 V}{2\pi}
+ c_2\, c_1^2\, V^{-\frac 32}\,
  \left(\frac{V\lambda}{c_1}\right)^{\frac
    32-\frac{2}{\sqrt{\log_2(V\lambda/c_1)}}}\,
\, \sum_{j=1}^{n}\, l_j\,
\Theta\left(\lambda-\frac{9}{d_j^2}\right)  \, .
\end{align}
Next we set $\lambda=\lambda_k$ and note that inequality
\eqref{li-yau-indiv} yields 
\begin{equation} \label{k-estimate}
\frac{2\pi}{V}\, k\, \leq \, \lambda_k\, .
\end{equation}
Since the right hand side of \eqref{pre-ineq} is an increasing
function of $\lambda$, we can use \eqref{k-estimate} to conclude that
\begin{align} \label{polyg-1}
\sum_{j=1}^k\, \lambda_j & \geq \frac{2\pi}{V}\,
k^2 + 4\, c_3\, k^{\frac 32-\frac{2}{\sqrt{\log_2(2\pi k/c_1)}}}\,
\sum_{j=1}^{n}\, l_j\, \Theta\left(k-\frac{9\, V}{2\pi\,
    d_j^2}\right)\, V^{-3/2}
\end{align}
where
$$
c_3 = \frac{2^{-3}}{9\sqrt{2}\, 36}\, (2\pi)^{\frac 54} c_1^{1/4}\, .
$$
Finally, we combine inequalities \eqref{polyg-1} and \eqref{lin} to
get \eqref{polyg}.
\end{proof}


\section{Proof for general domains}
\label{general-proof}
From now on we suppose that $\Omega$ is a general domain satisfying
assumption \ref{Omega}. To prove a Li-Yau type inequality with the
correction term we cannot directly employ the approach invented for 
polygons, since $\pd\Omega$ is in general nowhere straight. However,
we can extend 
$\Omega$ by adding small ``bumps'' to certain parts of $\pd\Omega$, see
Figure 4, in order to  obtain an extended domain $\Omega^e$ whose
boundary is in certain parts represented by a straight line. On these
straight pieces of $\pd\Omega^e$ we will
then employ the same strategy as in the case of polygons. Due to the
monotonicity of eigenvalues, any lower bound on the sum of the eigenvalues 
on the extended domain gives also a lower bound on the sum of the
eigenvalues on $\Omega$. On the other hand, we have to make sure that
the volume of $\Omega^e$ is not much bigger than $V$, because
otherwise it
could destroy the effect of the correction term in \eqref{main} by
decreasing the leading term. We will again split the exposition in
several steps.  

\subsection{Step 1: Some geometrical remarks}
\label{geom-remarks}

Here we will show that $\pd\Omega\cap\Gamma_j$ can be locally
represented as a graph of a certain $C^2-$smooth function. 
Let $\Gamma=\{x_1(s),x_2(s)\}$ be a part of the boundary of $\Omega$
parametrised by its length $s$ and such that $x_1(s),x_2(s)\in
C^2(\R_+)$. Let
$$
\varkappa_0: =\max_{\{x_1,x_2\}\in\Gamma}\, |\varkappa(x_1,x_2)|
$$
be the maximal curvature of $\Gamma$. We consider certain points
$A=\{x_1(s'),x_2(s')\}\in\Gamma$ and $B=\{x_1(s''),x_2(s'')\}\in\Gamma$
and chose a new system $(u,v)$ such that $A=(0,0)$ and the $u-$axes
goes along the line $AB$.

\begin{Lemma} \label{geom}
Assume that $\varkappa_0 |s'-s''| \leq \pi/4$. Then the following
statements hold true.
\begin{itemize}
\item[(i)] The part of $\Gamma$ connecting $A$ and $B$ can be written
in the system of coordinates $(u,v)$ as $v=v(u),\, u\in[0,u_0]$, where
$u_0=|AB|$. Moreover, we have
\begin{equation} \label{maximum}
\max_{u\in[0,u_0]}\, v(u) \leq \sqrt{2}\, \varkappa_0\, u_0^2\, .
\end{equation}
\item[(ii)] The inequality
\begin{equation} \label{sqrt2}
2^{-1/2}\, |s'-s''| \, \leq \, |AB|\, \leq\, |s'-s''|
\end{equation}
holds.
\end{itemize}
\end{Lemma}

\begin{proof}
Let $\{u(s),v(s)\}$ be the parametrisation of $\Gamma$ in the
coordinates $(u,v)$. By assumption we have
\begin{equation} \label{angle}
\int_0^{|s'-s''|}\, \varkappa(s)\, ds \leq \pi/4
\end{equation}
This means that for any $s\in[0,|s'-s''|]$ the angle between the tangent of
$\Gamma$ at the point $\{u(s),v(s)\}$ and the $u-$axes is less than or
equal to $\pi/4$.  Assume that there exists
$s_1,s_2\in[0,|s'-s''|]$ such that $u(s_1)=u(s_2)$. Then there exists
$s_3\in[s_1,s_2]$ such that the tangent of
$\Gamma$ at $\{u(s_3),v(s_3)\}$ is orthogonal to the $u-$axes. The
latter contradicts \eqref{angle}. This shows that the part of
$\Gamma$ between $A$ and $B$ can be considered as the graph of the function
$$
v=v(u),\quad u\in[0,u_0],\quad v(0)=v(u_0)=0\, .
$$
This proves the first part of $(i)$ and, in view of \eqref{angle},
shows that $|v'(u)|\leq 1$ on $[0,u_0]$.
Next we prove inequality \eqref{sqrt2}. It thus follows that
$$
u_0 = |AB| \leq |s'-s''| = \int_0^{u_0}\,
\left(1+|v'(u)|^2\right)^{1/2}\, du \leq 2^{1/2}\, u_0\, ,
$$
which implies \eqref{sqrt2}. To prove \eqref{maximum} we note that
$v(u)$ is twice
differentiable and therefore there exists some $u_1\in[0,u_0]$, such that
$v'(u_1)=0$. Since $|v''(u)| = |\varkappa(u)|\,
(1+|v'(u)|^2)^{3/2}\leq 2^{3/2}\, \varkappa_0$, we obtain
$$
|v'(u)| \leq \int_{u_1}^u\, |v''(u)|\, du\leq 2^{3/2}\, \varkappa_0\,
u_0\qquad \forall u\in[0,u_0] .
$$
The last inequality together with the fact that $v(0)=v(u_0)=0$
finally implies
$$
|v(u)| \leq \frac 12\,  2^{3/2}\, \varkappa_0\,
u_0^2 = 2^{1/2}\, \varkappa_0\, u_0^2 \qquad \forall
u\in[0,u_0]\, .
$$
\end{proof}


\subsection{Step 2: Approximation of the boundary}
\label{approx}
Next we introduce a procedure that allows us to choose appropriate 
parts of $\pd\Omega\cap\Gamma_j$ on which we will construct the additional 
``bumps'', see Figure 4. 
Let $\Gamma_j,\, j=1\dots m$ be the parts of boundary defined in
section \ref{mainresult} with the end points $A_j,B_j$ and the
partition $a_i^j,\, i=0,\dots,n_j$. We fix $j\in\{1,...,m\}$ and
take $\lambda$ large enough, such that
\begin{equation} \label{lambda-critical-1}
\lambda^{-\frac 12} \leq \min\left\{\frac{d_j}{3}\, ,\,
  \frac{\pi}{8\sqrt{2}\, \varkappa_j}\right\}, \quad
\text{if\, \, } L(\Gamma_j) > \frac{3\pi}{8\varkappa_j}\,
\end{equation}
and
\begin{equation} \label{lambda-critical-2}
\lambda^{-\frac 12} \leq \min\left\{\frac{d_j}{3}\, ,\,
\frac{L(\Gamma_j)}{3\sqrt{2}}\right\}, \quad
\text{if\, \, } L(\Gamma_j) \leq \frac{3\pi}{8\varkappa_j}\, .
\end{equation}
Let us consider $\Gamma_j\cap\, (a_i^j,a^j_{i+1})$ with $0<i<n_j$. On this
part of the boundary we choose several disjoint arcs $(b_l,b_l')$,
see Figure 3, such that each of them has the length $\sqrt{2}\,
\lambda^{-1/2}$ and such that
$$
\sum_l\, s(b_l,\, b_l') \geq \frac 13\, s(a^j_i,\, a^j_{i+1}), \quad
s(a^j_i,\, a^j_{i+1})-\sum_l\, s(b_l,\, b_l') \leq \, \sqrt{2}\,
\lambda^{-1/2}\, ,
$$
where $s(a,b)$ denotes the arc-length between $a$ and $b$.

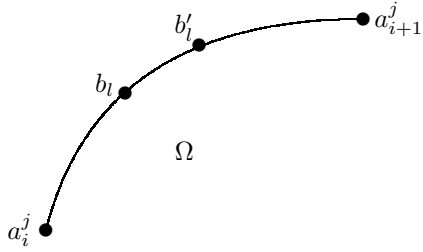
\begin{figure}[h] \label{tiling}
\begin{picture}(110,110)(-40,0)
\put(75,20){\makebox(0,0)[r]{$a_i^j$}}
\put(80,20){\circle*{5}}
\qbezier(80,20)(100,100)(200,100)
\put(223,100){\makebox(0,0)[r]{$a_{i+1}^j$}}
\put(200,100){\circle*{5}}
\put(110,72){\circle*{5}}
\put(106,75){\makebox(0,0)[r]{$b_l$}}
\put(138,90){\circle*{5}}
\put(136,97){\makebox(0,0)[r]{$b'_l$}}
\put(136,50){\makebox(0,0)[r]{$\Omega$}}
\end{picture}
\caption{Tiling of $\Gamma_j$.}
\end{figure}

Next we pick an $l$ and connect $b_l$ and $b_l'$ with a straight line
and choose a local
system of coordinates $(y_1,y_2)$ so that the $y_1-$axis goes
along the straight line from $b_l$ to $b_l'$ and the origin is in
$b_l$, see Figure 5. Notice that $s(a^j_{i-1},\, a^j_{i+1})=
s(a^j_i,\, a^j_{i+2}) \leq
\frac{\pi}{2\varkappa_j}$, which according to Lemma \ref{geom} means that
in the chosen coordinate system the boundary between $a^j_{i-1}$ and
$a^j_{i+2}$ can be written explicitly as $y_2=f(y_1)$. Let $y_0
=\text{dist}(b_l,b_l')$. In view of Lemma \ref{geom}
$$
\max_{y_1}\, |f(y_1)| \leq \sqrt{2}\, \varkappa_j\, y_0^2 \leq 2^{\frac
  32}\, \varkappa_j\, \lambda^{-1}\, .
$$

\noindent Now we introduce
$$
\Sigma_1= \left\{(y_1,y_2)\, :\, y_1\in[0,y_0], \, y_2 = 2^{\frac
  32}\, \varkappa_j\, \lambda^{-1}\right\}
$$
and
$$
\Sigma_2= \left\{(y_1,y_2)\, :\, y_1\in[0,y_0], \, y_2 = -2^{\frac
32}\, \varkappa_j\, \lambda^{-1}\right\}
$$

\begin{Lemma} \label{lemma2}
If $\lambda > 6\varkappa_j/d_j$, then
\begin{equation*}
\Sigma_1 \cap\, \pd\Omega= \Sigma_2\cap\, \pd\Omega= \emptyset\, .
\end{equation*}
\end{Lemma}

\begin{proof}
Obviously $\Sigma_1$ and $\Sigma_2$ do not cross $\pd\Omega$ between
$a^j_{i-1}$ 
and $a^j_{i+2}$. On the other hand, for each point $P=(y^P_1,y^P_2)$
holds
$$
\text{dist} (P,(a^j_{i},\, a^j_{i+1}))\, \leq \, 2^{3/2}\, \varkappa_j
\, \lambda^{-1}\, .
$$
Since $\text{dist} \left((a^j_{i},\, a^j_{i+1}),\, \pd\Omega\setminus
  (a^j_{i-1},\, a^j_{i+2})\right)\geq d_j$, this implies
$$
\text{dist} \left(P,\, \pd\Omega\setminus (a^j_{i-1},\,
  a^j_{i+2})\right)\geq d_j - 2^{3/2}\, \varkappa_j
\, \lambda^{-1} > \frac{d_j}{2} >0\, .
$$
\end{proof}
\noindent The last Lemma says that one of the sets $\Sigma_1$ and
$\Sigma_2$ is inside 
$\Omega$ and the other one is outside $\Omega$. Without loss of
generality we assume that $\Sigma_1$ is outside $\Omega$.

\subsection{Step 3: Extended domain $\Omega^e$.}
\label{ext}
The extended domain $\Omega^e$ differs from $\Omega$ if $\lambda$ is
large enough so that \eqref{lambda-critical-1} respectively
\eqref{lambda-critical-2} is satisfied 
(otherwise it coincides with $\Omega$).

\begin{figure}[h!]
\begin{picture}(110,110)(-50,0)
\put(95,42){\makebox(0,0)[r]{$b_l$}}
\put(105,70){\makebox(0,0)[r]{$\Omega^e$}}
\put(246,42){\makebox(0,0)[r]{$b'_l$}}
\qbezier(70,40)(170,100)(260,40)
\linethickness{1.4pt}
\qbezier(70,40)(80,80)(90,80)
\qbezier(260,40)(255,80)(245,80)
\multiput(90,80)(5,0){31}{\line(1,0){5}}
\thinlines
\put(260,40){\circle*{4}}
\put(70,40){\circle*{4}}
\put(75,30){\makebox(0,0)[r]{$\tilde{b}_l$}}
\put(265,30){\makebox(0,0)[r]{$\tilde{b}'_l$}}
\put(170,64){\makebox(0,0)[r]{$\pd\Omega$}}
\put(170,30){\makebox(0,0)[r]{$\Omega$}}
\multiput(90,51)(9,0){17}{\line(1,0){5}}
\put(90,51){\circle*{4}}
\put(243,51){\circle*{4}}
\put(250,90){\makebox(0,0)[r]{$\Sigma_1$}}
\end{picture}
\caption{Construction of the extended domain $\Omega^e$. The thick
  line represents the boundary of $\Omega^e$.}
\end{figure}
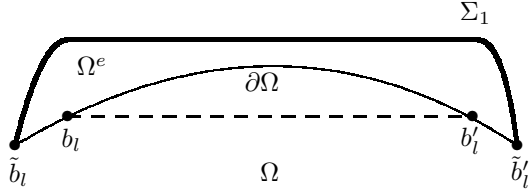

To define $\Omega^e$ we proceed as follows. For a fixed
$j\in\{1,\dots,m\}$, fixed $i\in\{1,...,n_j-1\}$ and fixed $l$,
we consider the boundary between the points
$b_l$ and $b_l'$. If it is a straight line, we do not change it. Otherwise we
replace this piece of the boundary with the segment $\Sigma_i$, where
$i$ is such that $\Sigma_i$ is outside $\Omega$, and connect the end
points of $\Sigma_1$ with
the boundary at certain points $\tilde{b}_l\in (b'_{l-1},b_l)$ and
$\tilde{b}'_l\in (b'_{l},b_{l+1})$ with appropriate $C^2$ functions,
see Figure 4. We choose
these function and the points $\tilde{b}_l,\, \tilde{b}_l'$ in such a
way that the
added area to $\Omega$ is less than $3$ times the area of the
rectangle with the corners given by $b_l$, $b_l'$ and the end
points of $\Sigma_1$. We then obtain a new region whose boundary,
corresponding to the original piece $\Gamma_j$ is again $C^2-$smooth
and which between the original boundary points $b_l$ and $b_l'$
consists of a straight line, see Figure 4.
Repeating this procedure for all $\Gamma_j,\, j=1,\dots,m$,
all $i\in\{1,\dots,n_j-1\}$ and all $l$ we thus obtain a new domain
$\Omega^e$.

As a next step we construct the squares $T_l$ of the side $\frac
12\, \lambda^{-1/2}$ between the the points $b_l$ and $b_l'$ centred
in the middle, see Figure 5. Note that, according to Lemma \ref{geom},
 $|b_l b_l'|\geq \lambda^{-1/2}/\sqrt{2}$. We have

\begin{Lemma} \label{squares}
The squares $T_l$ do not overlap.
\end{Lemma}

\begin{proof}
First we show that every $T_l$ does not overlap with any of the
squares constructed on the part of the boundary different from the
arch $(a^j_{i-1},\, a^j_{i+2})$. Indeed, each point of $T_l$ has distance
to $(b_l,b_l')$ at most $\frac 12\, \lambda^{-1/2}$ and the distance
between  $(b_l,b_l')$ and $\pd\Omega\setminus (a^j_{i-1},\, a^j_{i+2})$ is
at least $d_j$. Since $\lambda^{-1/2} < d_j$, see
\eqref{lambda-critical-1}, the result follows.

\begin{figure}[h]
\begin{picture}(110,110)(-50,0)
\put(95,42){\makebox(0,0)[r]{$b_l$}}
\put(246,42){\makebox(0,0)[r]{$b'_l$}}
\qbezier(70,40)(170,100)(260,40)
\put(67,40){\makebox(0,0)[r]{$\pd\Omega$}}
\put(120,25){\makebox(0,0)[r]{$\Omega$}}
\multiput(90,80)(10,0){16}{\line(1,0){5}}
\put(90,51){\vector(1,0){180}}
\put(90,51){\circle*{4}}
\put(243,51){\circle*{4}}
\put(90,50){\vector(0,1){50}}
\put(285,50){\makebox(0,0)[r]{$y_1$}}
\put(92,107){\makebox(0,0)[r]{$y_2$}}
\put(250,90){\makebox(0,0)[r]{$\Sigma_1$}}
\linethickness{1.4pt}
\put(138,80){\line(1,0){60}}
\put(138,80){\line(0,-1){60}}
\put(138,20){\line(1,0){60}}
\put(198,20){\line(0,1){60}}
\thinlines
\put(170,28){\makebox(0,0)[r]{$T_l$}}
\put(170,10){\vector(1,0){28}}
\put(170,10){\vector(-1,0){32}}
\put(190,0){\makebox(0,0)[r]{$\frac 12 \lambda^{-1/2}$}}
\end{picture}
\caption{The thick
  lines represent the square $T_l$. The dashed line represents the
  set $\Sigma_1$.}
\end{figure}
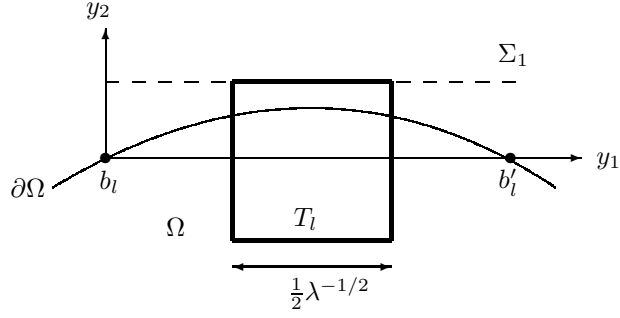

Consider now $(a^j_{i-1},\, a^j_{i+2})$. This part can be written as
$y_2=f(y_1)$ in the above introduced coordinate system. Consider
the squares $T_{l_1}$ and $T_{l_2}$ with $l_1\neq l_2$.
Let $y_1^1$ be the $y_1$ coordinate of the middle point
between $b_{l_1}$ and $b_{l_1}'$ and let $y^2_1$ be the $y_1$
coordinate of the middle point between $b_{l_2}$ and $b_{l_2}'$. Since
$|f'(y_1)| \leq 1$ on $(a^j_{i-1},\, a^j_{i+2})$, we have $|y^1_1-y^2_1|
\geq \lambda^{-1/2}$. For all points $(y_1,y_2)\in T_{l_1}$ holds
$|y_1-y^1_1| \leq \frac 14 \lambda^{-1/2}$ and for all points
$(y_1,y_2)\in T_{l_2}$ holds $|y_1-y^2_1| \leq \frac{\sqrt{2}}{2}\,
\lambda^{-1/2}$. Collecting these inequalities we conclude that
$T_{l_1}\cap T_{l_2} = \emptyset$.
\end{proof}

\noindent As a consequence of the last result we obtain estimates on the
volume of $\Omega^e$, which will be used in the proof of Theorem
\ref{maintheorem}.

\begin{Corollary} \label{V-extended}
Let $V^e$ be the volume of the extended domain
$\Omega^e$. Then
\begin{equation} \label{V-1}
V^e \ \leq \ V + 2^{\frac 32}\, \lambda^{-1} \sum_{j=1}^m\, \varkappa_j\,
L(\Gamma_j)\, .
\end{equation}
Moreover, if 
\begin{equation*}
\lambda \, \geq \, \Lambda_1 := 9\cdot 2^{10}\, \max_j\, \varkappa_j^2\,
,
\end{equation*}
then 
\begin{equation} \label{V-2}
V^e \ \leq \  2 V\, .
\end{equation}
\end{Corollary}

\begin{proof}
Inequality \eqref{V-1} follows directly from the construction of
$\Omega^e$, since the area of the added volume along $\Gamma_j$ does
not exceed $2^{\frac 32}\, \lambda^{-1}\, \varkappa_j L(\Gamma_j)$. 
As for the second inequality, we consider each pair $b_l,\, b_{l'}$ 
and note that for $\lambda\geq 9\cdot 2^{10}\, \varkappa_j^2$ is the
area of the added volume between $\tilde{b}_l$ and $\tilde{b}_{l'}$,
see Figure 4, bounded from above by 
$$
12 \, \varkappa_j \, \lambda^{-\frac 32}\, \leq \, \frac 18\, \lambda^{-1}\,
.
$$
This follows from the choice of the points $b_l$, see section
\ref{ext}. On the other hand, for $\lambda$ chosen as above we get   
$$
|T_l\cap \Omega| \geq \frac 12\, |T_l| = \frac 18\, \lambda^{-1}\, .
$$
Since $T_l$ do not overlap, we obtain \eqref{V-2}. 
\end{proof}

\begin{figure}[h]
\begin{picture}(110,110)(-50,0)
\put(95,42){\makebox(0,0)[r]{$b_l$}}
\put(246,42){\makebox(0,0)[r]{$b'_l$}}
\put(90,80){\vector(1,0){180}}
\qbezier(70,40)(80,80)(90,80)
\qbezier(260,40)(255,80)(245,80)
\put(90,51){\circle*{4}}
\put(243,51){\circle*{4}}
\put(70,33){\makebox(0,0)[r]{$\pd\Omega^e$}}
\put(270,90){\makebox(0,0)[r]{$x_2$}}
\multiput(90,51)(9,0){17}{\line(1,0){5}}
\linethickness{1.4pt}
\put(138,80){\line(1,0){60}}
\put(138,80){\line(0,-1){60}}
\put(138,20){\line(1,0){60}}
\put(198,20){\line(0,1){60}}
\thinlines
\put(167,80){\vector(0,-1){80}}
\put(167,-5){\makebox(0,0)[r]{$x_1$}}
\put(185,28){\makebox(0,0)[r]{$T_l$}}
\end{picture}
\caption{Construction of the local coordinate system at the boundary
  of $\Omega^e$.}
\end{figure}
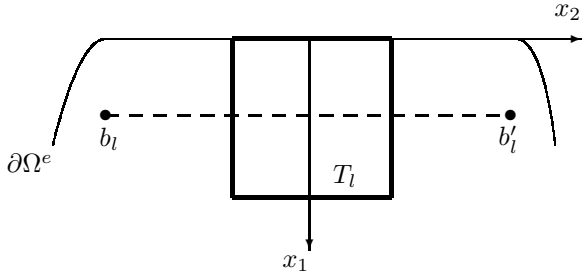


\subsection{Proof of Theorem \ref{maintheorem}}

\begin{proof}[Proof of Theorem \ref{maintheorem}]
Fix $\lambda >0$ and consider the extended domain $\Omega^e$.
Let $\mu_j$ be the eigenvalues of the Dirichlet Laplacian on $\Omega^e$
and let $\phi_j$ be the corresponding normalised
eigenfunctions. For $k\in\N$ we define
\begin{equation*}
F_e(\xi) = \sum_{j=1}^k\, |\hat{\phi}_j(\xi)|^2 \, ,
\end{equation*}
where $\hat{\phi_j}$ denotes the Fourier transform of $\phi_j$.
By $F^*_e(|\xi|)$ we denote the decreasing radial
rearrangement of $F_e(\xi)$. Let
$$
\phi(x)= \sum_{\mu_i \leq \lambda}\, c_i\, \phi_i(x),\quad \text{with\, \, }
\sum_{\mu_i \leq\lambda}|c_i|^2 \leq V^e 
$$
and let $T_l$ be the sequence of squares constructed along
$\Gamma_j$. For each $j$ is the number of these
squares at least
\begin{equation*}
N_j = \left [\frac{1}{9\sqrt{2}}\, L(\Gamma_j)\, \lambda^{\frac
  12}\right]\, .
\end{equation*}
Next we take $\lambda\geq \Lambda_1$, so that $V^e\leq 2V$, see
Corollary \ref{V-extended}. According to Corollary \ref{corol-rect}
for each $l$ and $p$ we then have 
\begin{equation*}
\left\|\frac{\pd^{p+1}\phi}{\pd \nu^{p+1}}\right\|^2_{L^2(R_n)}\,
\leq \, \frac{A_p(p) (V^e)^2}{4\pi}\, \lambda^{p+2} \, \leq 
 \frac{A_p(p) V^2}{\pi}\, \lambda^{p+2} \, ,
\end{equation*}
where $\frac{\pd\phi}{\pd \nu}$ denotes the normal derivative of
$\phi$.
In view of Proposition \ref{prop2} for each $l$ holds
\begin{equation*} 
\left\|\phi-e^{i\xi\cdot x} \right\|^2_{L^2(T_l)}\,
 \geq \frac{1}{36}\,\min\left\{4^{-p-\frac 52}\, \lambda^{-1},\,
  4^{-\frac p2-\frac 32}\, 6^{\frac{1}{2p}}\,
  (\beta^2_{p+1}+\beta^2_p)^{-\frac{1}{2p}}\, \lambda^{-1-\frac 1p}\right\},
\end{equation*}
with
$$
\beta^2_{p+1} = \frac{A_p(p) (V^{e})^2}{4\pi}\, \leq\,
\frac{A_p(p) V^2}{\pi}\, .
$$
Now we employ the same arguments used in the proof of Theorem
\ref{polygons} in order to find an appropriate upper bound on
$F_e^*$. Since $\lambda\geq \Lambda_1$ we can use Corollary
\ref{V-extended} to arrive at  
\begin{align*}
F^*_e(|\xi|) \leq \frac{V}{4\pi^{2}}\,
\left[1
  +\sum_{j=1}^m\left(2^{3/2}V^{-1}\varkappa_j\lambda^{-1}-\frac{c_2}{2}
V^{-\frac 12}\,  \left(\frac{V\lambda}{c_1}\right)^{-\frac
  12-\frac{2}{\sqrt{\log_2  (V\lambda/c_1)}}}\,\right) L(\Gamma_j)\, \right] .
\end{align*}
Note that for
$$
\lambda \geq \Lambda_2:= 2^{2^6}\, c_1\, V^{-1}
$$
we have $ \left(\frac{V\lambda}{c_1}\right)^{-\frac 12-\frac{2}{\sqrt{\log_2
(V\lambda/c_1)}}}\, \geq  \left(\frac{V\lambda}{c_1}\right)^{-\frac
34}$ and therefore
\begin{align*}
F^*_e(|\xi|)  \leq \ M_e(p,\lambda) 
:= \frac{V}{4\pi^{2}}\,
\left[1-\frac{c_2}{4}\, V^{-\frac 12}\,
  \left(\frac{V\lambda}{c_1}\right)^{-\frac
    12-\frac{2}{\sqrt{\log_2(V\lambda/c_1)}}}\,
\, \sum_{j=1}^{m}\, L(\Gamma_j)\,
\Theta(\lambda-\Lambda_3(j))\right]\, .
\end{align*}
where
$$
\Lambda_3(j) := \max\left\{\Lambda_1\, ,\, \Lambda_2\, , \,
  c_1^{-1}\, 2^{22}\, 6^8\, \varkappa_j^4\, V\right\}\, .
$$
We now use again the Li-Yau type minimiser
\eqref{minimizer-li-yau} with $V/4\pi^2$ replaced by
$M_e(p,\lambda)$ to obtain
\begin{equation*}
\sum_{j=1}^k\, \lambda_j \geq \sum_{j=1}^k\, \mu_j \geq
\int_{\R^2}F^*_e(|\xi|)|\xi|^2\, d\xi\, \geq \frac{\lambda^2 V^2}{8\pi^3\,
  M_e(p,\lambda)}\, .
\end{equation*}
As in the proof of Theorem \ref{polygons} 
we set $\lambda=\lambda_k$ and use definition of $M_e(p,\lambda)$
together with inequalities
 \eqref{lambda-critical-1},\eqref{lambda-critical-2} and
\eqref{k-estimate} to obtain
\begin{align} \label{second-ext}
\sum_{j=1}^k\, \lambda_j & \geq \frac{2\pi}{V}\,
k^2 + c_3\, k^{\frac 32-\frac{2}{\sqrt{\log_2(2\pi k/c_1)}}}\,
\sum_{j=1}^{m}\, L(\Gamma_j)\, \Theta(k-k(j))\, V^{-3/2}
\end{align}
where
\begin{equation*}
k(j):= \frac{V}{2\pi}\, \max\left\{\Lambda_3(j),\,
  \frac{9}{d^2_j},\, \frac{128\, \kappa^2_j}{\pi^2},\,
 \frac{6\varkappa_j}{d_j}\right\}\, .
\end{equation*}

\noindent 
Finally, we combine inequalities \eqref{second-ext} and \eqref{lin} to
get \eqref{main}.
\end{proof}



%
%
\providecommand{\bysame}{\leavevmode\hbox to3em{\hrulefill}\thinspace}
\providecommand{\MR}{\relax\ifhmode\unskip\space\fi MR }
\providecommand{\MRhref}[2]{%
  \href{http://www.ams.org/mathscinet-getitem?mr=#1}{#2}
}
\providecommand{\href}[2]{#2}

\section*{Acknowledgement}
The support from the DFG grant WE 1964/2 is gratefully acknowledged.

\end{document}